\documentclass[11pt]{article}
\usepackage{graphicx}
\usepackage{amsthm, amsmath, amssymb, tikz, bm}
\usepackage{mathtools}
\usepackage{xifthen}

\usepackage[margin=1.2in]{geometry} 

\usepackage[shortlabels]{enumitem}
\usepackage{todonotes}
\usepackage[colorlinks=true,
linkcolor=blue,citecolor=blue,
urlcolor=blue]{hyperref}

\usetikzlibrary{calc,shapes, backgrounds}

\newtheorem{lemma}{Lemma}
\newtheorem{theorem}{Theorem}
\newtheorem{corollary}{Corollary}

\newtheorem{claim}{Claim}
\newtheorem{remark}{Remark}

\newtheorem{proposition}{Proposition}
\newtheorem{question}{Question}

\newcommand{\ds}{\displaystyle}

\newcommand{\lp}{\left (}
\newcommand{\rp}{\right )}

\newcommand{\cH}{\mathcal{H}}

\newcommand{\ex}{\hat{\rm{ex}}}
\newcommand{\td}{\hat{\pi}}
\newcommand{\hh}{\hat{\mathcal{H}}}
\newcommand{\hd}{\hat{d}}

\DeclarePairedDelimiter{\abs}{\lvert}{\rvert}%

\newcommand{\Prob}[1]{\Pr\lp#1\rp}

\title{On the cover Tur\'an number of Berge hypergraphs}

\author{
Linyuan Lu
\thanks{University of South Carolina, Columbia, SC 29208,
({\tt lu@math.sc.edu}). This author was supported in part by NSF grant DMS-1600811.} \and
Zhiyu Wang \thanks{University of South Carolina, Columbia, SC 29208,
({\tt zhiyuw@math.sc.edu}).} 
}

\begin{document}

\maketitle

\begin{abstract}
For a fixed set of positive integers $R$, we say $\cH$ is an $R$-uniform hypergraph, or $R$-graph, if the cardinality of each edge belongs to $R$. For a graph $G=(V,E)$, a hypergraph $\mathcal{H}$ is called a \textit{Berge}-$G$, denoted by $BG$, if there is
an injection $i\colon V(G)\to V(\mathcal{H})$ and 
a bijection $f\colon E(G) \to E(\mathcal{H})$ such that for all $e=uv \in E(G)$, we have $\{i(u), i(v)\} \subseteq f(e)$.
%if there exists a bijection $f: E(G) \to E(\mathcal{H})$ such that for every $e \in E(G)$, $e \subseteq f(e)$. 
In this paper, we define a variant of Tur\'an number in hypergraphs, namely the \emph{cover Tur\'an number}, denoted as $\hat{ex}_R(n, G)$, as the maximum number of edges in the shadow graph of a Berge-$G$ free $R$-graph on $n$ vertices. We show a general upper bound on the cover Tur\'an number of graphs and determine the cover Tur\'an density of all graphs when the uniformity of the host hypergraph equals to $3$. 
\end{abstract}

\section{Introduction}
A hypergraph is a pair $\cH=(V,E)$ where $V$ is a vertex set and $E\subseteq 2^V$ is an edge set. For a fixed set of positive integers $R$, we say $\cH$ is an $R$-uniform hypergraph, or $R$-graph for short, if the cardinality of each edge belongs to $R$. 
If $R=\{k\}$, then an $R$-graph is simply a $k$-uniform hypergraph or a $k$-graph.
Given an $R$-graph $\cH = (V,E)$ and a set $S \in \binom{V}{s}$, let $d(S)$ denote the number of edges containing $S$ and $\delta_s(\cH)$ be the minimum $s$-degree of $\cH$, i.e., the minimum of $d(S)$ over all $s$-element sets $S \in \binom{V}{s}$. When $s = 2$, $\delta_2(\cH)$ is also called the minimum \textit{co-degree} of $\cH$. Given a hypergraph $\cH$, the \textit{$2$-shadow}(or \textit{shadow}) of $\cH$, denoted by $\partial(\cH)$, is a simple $2$-uniform graph $G=(V,E)$ such that $V(G)=V(\cH)$ and $uv\in E(G)$ if and only if $\{u,v\}\subseteq h$ for some $h\in E(\cH)$. Note that $\delta_2(\cH) \geq 1$ if and only if $\partial(\cH)$ is a complete graph. In this case, we say $\cH$ is {\em covering}.

There are several notions of a path or a cycle in hypergraphs. A \textit{Berge path} of length $t$ is a collection of $t$ hyperedges $h_1, h_2, \ldots, h_{t} \in E$ and $t+1$ vertices $v_1, \ldots, v_{t+1}$ such that $\{v_i, v_{i+1}\} \subseteq h_i$ for each $i\in [t]$. Similarly, a $k$-graph $\cH = (V,E)$ is called a \textit{Berge} cycle of length $t$ if $E$ consists of $t$ distinct edges $h_1, h_2, \ldots, h_t$ and $V$ contains $t$ distinct vertices $v_1, v_2, \ldots, v_t$ such that $\{v_i, v_{i+1}\} \subseteq h_i$ for every $i\in [t]$ where $v_{t+1} \equiv v_1$. Note that there may be other vertices than $v_1, \ldots, v_t$ in the edges of a Berge cycle or path. Gerbner and Palmer \cite{GP} extended the definition of Berge paths and Berge cycles to general graphs. In particular, given  a simple graph $G$, a hypergraph $\cH$ is called  \emph{Berge-$G$} if there is
an injection $i\colon V(G)\to V(\cH)$ and 
a bijection $f\colon E(G) \to E(\cH)$ such that for all $e=uv \in E(G)$, we have $\{i(u), i(v)\} \subseteq f(e)$.

We say an $R$-graph $\cH$ on $n$ vertices contains a \textit{Hamiltonian Berge cycle (path)} if it contains a Berge cycle (path) of length $n$ (or $n-1$). We say $\cH$ is {\em Berge-Hamiltonian} if it contains a Hamiltonian Berge cycle. Bermond, Germa, Heydemann, and Sotteau \cite{BGHS} showed a Dirac-type theorem for Berge cycles. We showed in \cite{LW-codegree} that for every finite set $R$ of positive integers, there is an integer $n_0=n_0(R)$ such that every covering $R$-uniform hypergraph $\cH$ on $n$ ($n\geq n_0$) vertices contains a Berge cycle $C_s$ for any $3\leq s\leq n$. In particular, every covering $R$-graph on sufficiently large $n$ vertices is Berge-Hamiltonian.

Extremal problems related to Berge hypergraphs have been receiving increasing attention lately. For Tur\'an-type results, let $ex_k(n,G)$ denote the maximum number of hyperedges in $k$-uniform Berge-$G$-free hypergraph.
Gy\H{o}ri, Katona and Lemons \cite{GKL} showed that for a $k$-graph $\cH$ containing no Berge path of length $t$, if $t\geq k+2 \geq 5$, then $e(\cH) \leq \frac{n}{t}\binom{t}{k}$; if $3\leq t\leq k$, then $e(\cH) \leq \frac{n(t-1)}{k+1}$. Both bounds are sharp. The remaining case of $t=k+1$ was settled by Davoodi, Gy\H{o}ri, Methuku and Tompkins \cite{Davoodi}. For cycles of a given length, Gy\H{o}ri and Lemons \cite{GL1, GL2} showed that $ex_k(n, C_{2t}) = \Theta(n^{1+1/t})$. The same asymptotic upper bound holds for odd cycles of length $2t+1$ as well. The problem of avoiding all Berge cycles of length at least $k$ has been investigated in a series of papers \cite{KL, FKL, FKL2, Ergemlidze, GLSZ}. For general results on the maximum size of a Berge-$G$-free hypergraph for an arbitrary graph $G$, see for example \cite{GMP, GMT, PTTW}.

For Ramsey-type results, define $R_c^k(BG_1, \ldots, BG_c)$ as the smallest integer $n$ such that  for any $c$-edge-coloring of a complete $k$-uniform hypergraph on $n$ vertices, there exists a Berge-$G_i$ subhypergraph with color $i$ for some $i$. Salia, Tompkins, Wang and Zamora \cite{STWZ} showed that $R_2^3(BK_s, BK_t) = s+t -3$ for $s,t \geq 4$ and $\max(s,t) \geq 5$. For higher uniformity, they showed that $R^4(BK_t, BK_t) = t+1$ for $t \geq 6$ and $R_2^k(BK_t, BK_t)=t$ for $k\geq 5$ and $t$ sufficiently large.
Independently and more generally, Gerbner, Methuku, Omidi and Vizer \cite{GMOV} showed that $R_c^k(BK_n) = n$ if $k > 2c$; $R_c^k(BK_n) = n+1$ if $k = 2c$ and obtained various bounds on $R_c^k(BK_n)$ when $k < 2c$. Similar investigations have also been started independently by Axenovich and Gy\'arf\'as \cite{AG} who focus on the Ramsey number of small fixed graphs where the number of colors may go to infinity. 

Very recently, we \cite{LW-cRamsey} defined a new type of Ramsey number, namely the \emph{cover Ramsey number}, denoted as $\hat{R}^R(BG_1, BG_2)$, as the smallest integer $n_0$ such that for every covering $R$-uniform hypergraph $\mathcal{H}$ on $n \geq n_0$ vertices and every $2$-edge-coloring (blue and red) of $\mathcal{H}$, there is either a blue Berge-$G_1$ or a red Berge-$G_2$ subhypergraph. We show that for every $k\geq 2$, $R(G_1, G_2) \leq \hat{R}^{[k]}(BG_1, BG_2) \leq c_k \cdot R(G_1, G_2)^3$ for some $c_k$. Moreover, $R^{\{k\}}(K_t) >  (1+o(1))\frac{\sqrt{2}}{e} t2^{t/2}$ for sufficiently large $t$ and $\hat{R}^{[k]}(BG,BG) \leq c(d,k)n$ if $\Delta(G)\leq d$. It occurs to us that the cover Ramsey number for Berge hypergraphs behaves more like the classical Ramsey number than the Ramsey number of Berge hypergraphs defined in \cite{AG, STWZ, GMOV}. This inspires us to extend the investigation to the analogous \emph{cover Tur\'an number} for Berge hypergraphs. In particular, given a fixed graph $G$ and a finite set of positive integers $R \subseteq [k]$, we define the \emph{$R$-cover Tur\'an number} of $G$, denoted as \emph{$\ex_R(n, G)$}, as the maximum number of edges in the shadow graph of a Berge-$G$-free $R$-graph on $n$ vertices. The \emph{$R$-cover Tur\'an density}, denoted as $\td_R(G)$, is defined as 
\[\td_R(G) = \ds\limsup_{n\to \infty} \frac{\ex_R(n, G)}{\binom{n}{2}}.\]
When $R$ is clear from the context, we ignore $R$ and use \emph{cover Tur\'an number} and \emph{cover Tur\'an density} for short. 
 A graph is called \textit{$R$-degenerate} if $\td_R(G) = 0$. For the ease of reference, when $R=\{k\}$, we simply denote $\td_R(G)$ as $\td_k(G)$ and call $G$ \textit{$k$-degenerate} if $\td_{\{k\}}(G) = 0$.

We remark that the Tur\'an number of graphs only differ by a constant factor when the host hypergraph is uniform compared to non-uniform. In particular, we show the following proposition.

\begin{proposition}\label{prop:uniform-versus-non}
If $R$ is a finite set of positive integers such that $min(R) = m \geq 2$ and $max(R) = M$. Then given a fixed graph $G$,
$$\max_{r\in R}\ex_r(n,G) \leq \ex_{R}(n,G) \leq \frac{\binom{M}{2}}{\binom{m}{2}} \ex_m(n,G).$$
\end{proposition}
Indeed, the first inequality is clear from definition. For the second inequality, suppose we have an $R$-graph $\cH$ with more than $\binom{M}{2}/\binom{m}{2} \cdot \ex_m(n,G)$ edges in its shadow. 
For each hyperedge $h$ in $\cH$, shrink it to a hyperedge of size $m$ by uniformly and randomly picking $m$ vertices in $h$. Call the resulting hypergraph $\cH'$. It is easy to see that for any edge $e \in E(\partial(\cH))$, $\Prob{e \in E(\partial(\cH'))} \geq \binom{m}{2}/\binom{M}{2}$. Hence by linearity of expectation, the expected number of edges in $\partial(\cH')$ is more than $\ex_m(n,G)$. It follows that there exists a way to shrink $\cH$ to a $m$-graph with at least $(\ex_m(n,G)+1)$ edges in its shadow. Thus, by definition of the cover Tur\'an number, $\cH'$ contains a Berge copy of $G$, which corresponds to a Berge-$G$ in $\cH$.

\begin{remark}
Note that Proposition \ref{prop:uniform-versus-non} implies that if a graph $G$ is $k$-degenerate (where $k\geq 2$), then it is $R$-degenerate for any finite set $R$ satisfying $min(R) \geq k$. In particular, a bipartite graph is $k$-degenerate for all $k\geq 2$.
\end{remark}

In this paper, we determine the cover Tur\'an density of all graphs when the uniformity of the host graph equals to $3$. 
We first establish a general upper bound for the cover Tur\'an density of graphs.

\begin{theorem}\label{thm:general-upper}
For any fixed graph $G$ and any fixed $\epsilon > 0$, there exists $n_0$ such that for any $n\geq n_0$, 
$$\ex_{k}(n,G) \leq \lp  1- \frac{1}{\chi(G)-1} + \epsilon \rp \binom{n}{2}.$$
\end{theorem}

We remark that Theorem \ref{thm:general-upper} holds when the host hypergraph is non-uniform as well, i.e. we can replace $k$ with any fixed finite set of positive integers $R$.
If $\chi(G)>k$, there is a construction giving the matching lower bound. Partition the vertex set into $t:=\chi(G)-1$ equitable parts $V=V_1\cup V_2 \cup \cdots \cup V_t$. Let $\cH$ be the $k$-uniform hypergraph on the vertex set $V$ consisting of all $k$-tuples intersecting each $V_i$ on at most one vertex.  
The shadow graph is simply the Tur\'an graph with 
$ (1- \frac{1}{\chi(G)-1}+o(1))\binom{n}{2}$ edges.
The shadow graph is $K_{t+1}$-free, thus contains no subgraph $G$. It follows that $\cH$ is Berge-$G$-free. Therefore, we have the following theorem:

\begin{theorem}\label{thm:general}
For any $k\geq 2$, and any fixed graph $G$ with
$\chi(G)\geq k+1$, we have
$$\td_{k}(G) = 1- \frac{1}{\chi(G)-1}.$$
\end{theorem}

Given a simple graph $G$ on $n$ vertices $v_1, \ldots, v_n$ and a sequence of $n$ positive integers $s_1, \ldots, s_n$, we denote $B = G(s_1, \ldots, s_n)$ the $(s_1, \ldots, s_n)$-\emph{blowup} of $G$ obtained by replacing every vertex $v_i\in G$ with an independent set $I_i$ of $s_i$ vertices, and by replacing every edge $(v_i, v_j)$ of $G$ with a complete bipartite graph connecting the independent sets $I_i$ and $I_j$. If $s = s_1 = s_2 = \cdots = s_n$, we simply write $G(s_1, \ldots, s_n)$ as $G(s)$ where $s$ is called the \emph{blowup factor}.
We also define a \emph{generalized blowup} of $G$, denoted by $G(s_1, \ldots, s_n; M)$ where $M \subseteq E(G) \subseteq \binom{[n]}{2}$, as the graph obtained by replacing every vertex $v_i\in G$ with an independent set $I_i$ of $s_i$ vertices, and by replacing every edge $(v_i, v_j)$ of $E(G)\backslash M$ with a complete bipartite graph connecting $I_i$ and $I_j$ and replacing every edge $(v_i, v_j) \in M$ with a maximal matching connecting $I_i$ and $I_j$. When $M = \emptyset$, we simply write $G(s_1, \ldots, s_n; M)$ as the standard blowup $G(s_1, \ldots, s_n)$.

We first want to characterize the class of degenerate graphs when the host hypergraph is $3$-uniform. Observe that $\ex_k(n,G)\leq \binom{k}{2} ex_k(n, G)$. This implies that any graph $G$ satisfying $ex_k(n,G) = o(n^2)$ is $k$-degenerate. In particular, by results of \cite{GL1, GL2, GP, PTTW}, any cycles of fixed length at least $4$ and $K_{2,t}$ are $3$-degenerate. For triangles, Gr\'osz, Methuku and Tompkins \cite{GMT} showed that the uniformity threshold of a triangle is $5$, which implies that $C_3$ is $5$-degenerate. Moreover, there are constructions which show that $C_3$ is not $3$-degenerate or $4$-degenerate. For $K_{s,t}$ where $s,t\geq 3$, it is shown \cite{PTTW, GMT, AS} that $ex_r(n,K_{s,t}) = \Theta(n^{r-\frac{r(r-1)}{2s}})$. Thus in this case, the corresponding results on Berge Tur\'an number do not imply the degeneracy of $K_{s,t}$ in the cover Tur\'an density. 

In this paper, 
we classify all degenerate graphs when the host hypergraph is $3$-uniform.

\begin{theorem}\label{thm:degenerate}
Given a simple graph $G$,
$\td_3(G) = 0$ if and only if $G$ satisfies both of the following conditions:
\begin{enumerate}[(1)]
    \item $G$ is triangle-free, and there exists an induced bipartite subgraph $B \subseteq G$ such that $V(G)-V(B)$ is a single vertex.
    
    \item There exists a bipartite subgraph $B \subseteq G$ such that $E(G)-E(B)$ is a matching (possibly empty) in one of the partitions of $B$.
\end{enumerate}

\end{theorem}

\begin{corollary}\label{cor:degenerate-characterize1}
Given a simple graph $G$, $\td_3(G) = 0$ if and only if $G$ is contained in both $C_5(1,s,s,s,s)$ and $C_3(s,s,s;\{\{1,2\}\})$ for some positive integer $s$.
\end{corollary}

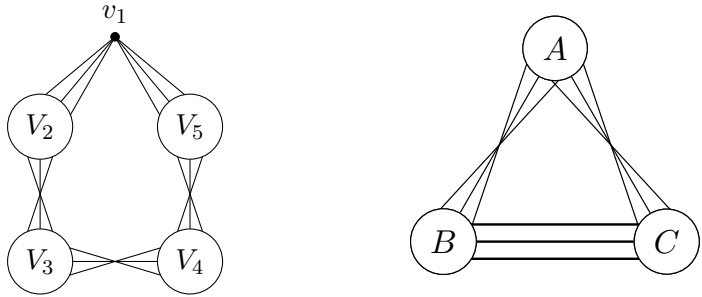
\begin{figure}[htb]
	\begin{center}
    \begin{minipage}{.2\textwidth}
		\resizebox{3cm}{!}{\begin{tikzpicture}[scale=1, Wvertex/.style={circle, draw=black, fill=white, scale=1}, bvertex/.style={circle, draw=black, fill=black, scale=0.3}]

\node [bvertex] (v1) at (0, 3) {};
\node [Wvertex] (v2) at (-1,1.8) {};
\node [Wvertex] (v3) at (-1,0) {};
\node [Wvertex] (v4) at (1,0 ) {};
\node [Wvertex] (v5) at (1,1.8) {};

\draw (v1) -- ($(v2) + (-0.2,0.2)$);
\draw (v1) -- (v2);
\draw (v1) -- ($(v2) + (0.2,-0.2)$);

\draw (v1) -- ($(v5) + (0.2,0.2)$);
\draw (v1) -- (v5);
\draw (v1) -- ($(v5) + (-0.2,-0.2)$);

\draw  ($(v2) + (-0.3,0)$) -- ($(v3) + (0.3,0)$);
\draw  (v2) -- (v3);
\draw  ($(v2)+(0.3,0)$) -- ($(v3) + (-0.3, 0)$);

\draw  ($(v4) + (-0.3,0)$) -- ($(v5) + (0.3,0)$);
\draw  (v4) -- (v5);
\draw  ($(v4)+(0.3,0)$) -- ($(v5) + (-0.3, 0)$);

\draw  ($(v3) + (0,0.3)$) -- ($(v4) + (0,-0.3)$);
\draw  (v3) -- (v4);
\draw  ($(v3)+(0, -0.3)$) -- ($(v4) + (0,0.3)$);

\node [bvertex, label=above:$v_1$] (v1) at (0, 3) {};
\node [Wvertex] (v2) at (-1,1.8) {$V_2$};
\node [Wvertex] (v3) at (-1,0) {$V_3$};
\node [Wvertex] (v4) at (1,0 ) {$V_4$};
\node [Wvertex] (v5) at (1,1.8) {$V_5$};

\end{tikzpicture}}
	\end{minipage}
    \hspace{2cm}
	\begin{minipage}{.2\textwidth}
		\resizebox{4cm}{!}{\begin{tikzpicture}[scale=1, Wvertex/.style={circle, draw=black, fill=white, scale=1}, bvertex/.style={circle, draw=black, fill=black, scale=0.3}]

\node [Wvertex] (v1) at (-1.3, 0) {$B$};
\node [Wvertex] (v2) at (1.3,0) {$C$};
\node [Wvertex] (v3) at (0,1.732*1.3) {$A$};

\draw[thick]  ($(v1)+(0,0.2)$) -- ($(v2) + (0,0.2)$);
\draw[thick] (v1) -- (v2);
\draw[thick] ($(v1)+(0,-0.2)$) -- ($(v2) + (0,-0.2)$);

\draw  ($(v1) + (-0.2,0.2)$) -- ($(v3) + (0.2,-0.2)$);
\draw  (v1) -- (v3);
\draw  ($(v1)+(0.2,-0.2)$) -- ($(v3) + (-0.2, 0.2)$);

\draw  ($(v2) + (-0.2,-0.2)$) -- ($(v3) + (0.2,0.2)$);
\draw  (v2) -- (v3);
\draw  ($(v2) + (0.2,0.2)$) -- ($(v3) + (-0.2, -0.2)$);

\node [Wvertex] (v1) at (-1.3, 0) {$B$};
\node [Wvertex] (v2) at (1.3,0) {$C$};
\node [Wvertex] (v3) at (0,1.732*1.3) {$A$};
\end{tikzpicture}}
	\end{minipage}
    \end{center}
    \caption{$C_5(1,s,s,s,s)$ and $C_3(s,s,s;\{\{1,2\}\})$}
    \label{fig:embed2}
\end{figure}

\begin{corollary}\label{cor:degenerate-characterize2}
Given a simple graph $G$, $\td_3(G) = 0$ if and only if $G$ is a subgraph of one of the graphs in Figure \ref{fig:embed3}.
\end{corollary}

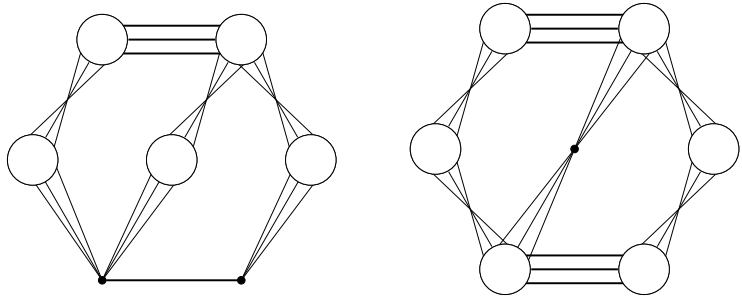
\begin{figure}[htb]
	\begin{center}
        \begin{minipage}{.2\textwidth}
        		\resizebox{4.5cm}{!}{\begin{tikzpicture}[scale=1, Wvertex/.style={circle, draw=black, fill=white, scale=2}, bvertex/.style={circle, draw=black, fill=black, scale=0.3}]

\node [bvertex] (v1) at (-1, -1.732) {};
\node [bvertex] (v2) at (1, -1.732) {};
\node [Wvertex] (v0) at (0,0) {};
\node [Wvertex] (v3) at (2,0) {};
\node [Wvertex] (v4) at (1,1.732) {};
\node [Wvertex] (v5) at (-1,1.732) {};
\node [Wvertex] (v6) at (-2,0) {};

\draw[thick] (v1)  -- (v2);
\draw (v1)-- ($(v0) + (-0.2,0)$);
\draw (v1)-- ($(v0) + (0,0)$);
\draw (v1)-- ($(v0) + (0.3,0)$);

\draw (v2)-- ($(v3) + (-0.2,0)$);
\draw (v2)-- ($(v3) + (0,0)$);
\draw (v2)-- ($(v3) + (0.3,0)$);

\draw (v1)-- ($(v6) + (-0.2,0)$);
\draw (v1)-- ($(v6) + (0,0)$);
\draw (v1)-- ($(v6) + (0.3,0)$);

\draw[thick]  ($(v5)+(0,0.2)$) -- ($(v4) + (0,0.2)$);
\draw[thick] (v5) -- (v4);
\draw[thick] ($(v5)+(0,-0.2)$) -- ($(v4) + (0,-0.2)$);

\draw  ($(v5)+(-0.2,0.2)$) -- ($(v6) + (0.2,-0.2)$);
\draw (v5) -- (v6);
\draw  ($(v5)+(0.2,-0.2)$) -- ($(v6) + (-0.2, 0.2)$);

\draw  ($(v0)+(-0.2,0.2)$) -- ($(v4) + (0.2,-0.2)$);
\draw (v0) -- (v4);
\draw  ($(v0)+(0.2,-0.2)$) -- ($(v4) + (-0.2, 0.2)$);

\draw  ($(v4)+(-0.2,-0.2)$) -- ($(v3) + (0.2,0.2)$);
\draw (v4) -- (v3);
\draw  ($(v4)+(0.2,0.2)$) -- ($(v3) + (-0.2, -0.2)$);

\node [Wvertex] (v0) at (0,0) {};
\node [Wvertex] (v3) at (2,0) {};
\node [Wvertex] (v4) at (1,1.732) {};
\node [Wvertex] (v5) at (-1,1.732) {};
\node [Wvertex] (v6) at (-2,0) {};
\end{tikzpicture}	}
        \end{minipage}
            \hspace{2cm}
        \begin{minipage}{.2\textwidth}
        		\resizebox{4.5cm}{!}{\begin{tikzpicture}[scale=1, Wvertex/.style={circle, draw=black, fill=white, scale=2}, bvertex/.style={circle, draw=black, fill=black, scale=0.3}]

\node [Wvertex] (v1) at (-1, -1.732) {};
\node [Wvertex] (v2) at (1, -1.732) {};
\node [bvertex] (v0) at (0,0) {};
\node [Wvertex] (v3) at (2,0) {};
\node [Wvertex] (v4) at (1,1.732) {};
\node [Wvertex] (v5) at (-1,1.732) {};
\node [Wvertex] (v6) at (-2,0) {};

\draw (v0)-- ($(v1) + (-0.2,0.2)$);
\draw (v0)-- ($(v1) + (0,0)$);
\draw (v0)-- ($(v1) + (0.2,-0.2)$);

\draw (v0)-- ($(v4) + (0.2,-0.2)$);
\draw (v0)-- ($(v4) + (0,0)$);
\draw (v0)-- ($(v4) + (-0.2,0.2)$);

\draw[thick]  ($(v5)+(0,0.2)$) -- ($(v4) + (0,0.2)$);
\draw[thick] (v5) -- (v4);
\draw[thick] ($(v5)+(0,-0.2)$) -- ($(v4) + (0,-0.2)$);

\draw[thick]  ($(v1)+(0,0.2)$) -- ($(v2) + (0,0.2)$);
\draw[thick] (v1) -- (v2);
\draw[thick] ($(v1)+(0,-0.2)$) -- ($(v2) + (0,-0.2)$);

\draw  ($(v5)+(-0.2,0.2)$) -- ($(v6) + (0.2,-0.2)$);
\draw (v5) -- (v6);
\draw  ($(v5)+(0.2,-0.2)$) -- ($(v6) + (-0.2, 0.2)$);

\draw  ($(v3)+(-0.2,0.2)$) -- ($(v2) + (0.2,-0.2)$);
\draw (v3) -- (v2);
\draw  ($(v3)+(0.2,-0.2)$) -- ($(v2) + (-0.2, 0.2)$);

\draw  ($(v4)+(-0.2,-0.2)$) -- ($(v3) + (0.2,0.2)$);
\draw (v4) -- (v3);
\draw  ($(v4)+(0.2,0.2)$) -- ($(v3) + (-0.2, -0.2)$);

\draw  ($(v6)+(-0.2,-0.2)$) -- ($(v1) + (0.2,0.2)$);
\draw (v6) -- (v1);
\draw  ($(v6)+(0.2,0.2)$) -- ($(v1) + (-0.2, -0.2)$);

\node [Wvertex] (v1) at (-1, -1.732) {};
\node [Wvertex] (v2) at (1, -1.732) {};
%\node [bvertex] (v0) at (0,0) {};
\node [Wvertex] (v3) at (2,0) {};
\node [Wvertex] (v4) at (1,1.732) {};
\node [Wvertex] (v5) at (-1,1.732) {};
\node [Wvertex] (v6) at (-2,0) {};

\end{tikzpicture}}
        \end{minipage}
    \end{center} 
    \caption{Characterization of $3$-degenerate graphs.}
    \label{fig:embed3}
\end{figure}

With Theorem \ref{thm:general-upper} and Theorem \ref{thm:degenerate}, we can then determine the cover Tur\'an density of all graphs when $k = 3$. The results are summarized in the following theorem.

\begin{theorem}\label{thm:cover-density-3}
Given a simple graph $G$,
    \[  \td_3(G) =   
        \begin{cases} 
         1-\frac{1}{\chi(G)-1}  & \textrm{if $\chi(G)\geq 4$,}\\
          0 & \textrm{if $G$ satisfies the condition in Theorem \ref{thm:degenerate},} \\
          \frac{1}{2} & \textrm{otherwise.}
       \end{cases}
    \]
\end{theorem}

For $3$-cover Tur\'an number, we also show the following proposition:

\begin{proposition}\label{prop:cover-number}
Let $G$ be a connected bipartite graph such that every edge is contained in a $C_4$ and every two vertices in the same part have a common neighbor.
Then 
$$\ex_3(n,G) = \Theta(ex(n,G)).$$
\end{proposition}
\begin{proof}
The fact that $\ex_3(n,G) = O(ex(n,G))$ is a consequence of Proposition \ref{prop:uniform-versus-non}. For the lower bound, consider an extremal $G$-free graph $H$ with $ex(n,G)$ edges. It follows that there is a bipartite subgraph $H'= A\cup B$ of $H$ which is $G$-free and contains at least $\frac{1}{2}ex(n,G)$ edges. We then construct a $3$-graph $\cH$ as follows. For each $a\in A$, replace $a$ with two new vertices $a_1, a_2$. The vertex set $B$ remains the same. For each $e = \{a,b\} \in E(H')$ with $a \in A$, $b\in B$, we have a hyperedge $\{a_1, a_2, b\}$ in $\cH$. We claim that $\cH$ contains no Berge-$G$. Indeed, if there is any Berge-$G$ in $\cH$, then one of the following two cases must happen:
\begin{description}
    \item Case 1: An edge in $G$ is mapped to $\{a_1, a_2\}$ for some $a \in A$. However, note that there is no $C_4$ containing $a_1 a_2$ in $\partial(\cH)$ while every edge of $G$ is contained in a $C_4$. This gives us a contradiction.
     \item Case 2: Two vertices of $G$ from the same part are mapped to $\{a_1, a_2\}$ for some $a \in A$. In this case, by our assumption, $a_1, a_2$ have a common neighbor $w$ in $G$. However, there are no two distinct hyperedges embedding $a_1w , a_2 w$ by our construction. Contradiction.
\end{description}
Hence it follows that $\cH$ is Berge-$G$-free and has $\Omega(ex(n,G))$ hyperedges.
\end{proof}

\begin{remark}
We give a class of graphs satisfying the conditions in Proposition \ref{prop:cover-number}. Let $B = B_1\cup B_2$ be an arbitrary connected bipartite graph with minimum degree $2$ such that each part has a vertex that is adjacent to all the vertices in the other part. It's easy to check that $B$ satisfies the conditions in  Proposition \ref{prop:cover-number}.
\end{remark}

Using Proposition \ref{prop:cover-number}, we have the following corollary on the asymptotics of the cover Tur\'an number of $K_{s,t}$.
\begin{corollary}
For positive integers $t\geq s\geq 2$, we have
$$\ex_3(n,K_{s,t}) = \Theta(ex(n,K_{s,t})).$$
\end{corollary}

The following questions would be interesting for further investigations:
\begin{question}
Characterize all $k$-degenerate graphs or determine the \{$k$\}-cover Tur\'an density of all graphs for $k\geq 4$.
\end{question}

\begin{question}
Determine the asymptotics of the cover Tur\'an number of the $3$-degenerate graphs in Theorem \ref{thm:degenerate}.
\end{question}

%%%%%%%%%%%%%%%%%%%%%%%%%%%%%%%%%%%%%%%%%%%%%%
% Proof of General Upper Bound
%%%%%%%%%%%%%%%%%%%%%%%%%%%%%%%%%%%%%%%%%%%%%%

\section{Proof of Theorem \ref{thm:general-upper}}

\begin{proof}[Proof of Theorem \ref{thm:general-upper}]
Let $k \geq 2$ and $G$ be a fixed graph with $\chi(G) \geq 2$. Let $\epsilon > 0$. Suppose $\cH$ is an edge-minimal $k$-uniform hypergraph on sufficiently large $n$ vertices such that 
$$\abs*{E(\partial(\cH))} \geq  \lp  1- \frac{1}{\chi(G)-1} + \epsilon \rp \binom{n}{2}.$$
Our goal is to show that $\cH$ contains a Berge copy of $G$.
For ease of reference, set $H = \partial(\cH)$. Let $M = k^2/\epsilon$. Let $H'$ be the subgraph of $H$ obtained by deleting all the edges $uv$ from $H$ with co-degree $d(\{u,v\}) \geq M$ in $\cH$. 

\begin{claim}
$\abs*{E(H')} \geq  \lp  1- \frac{1}{\chi(G)-1} + \epsilon/2 \rp \binom{n}{2}$.
\end{claim}
\begin{proof}
Let $L = E(H)\backslash E(H')$. By double counting, the number of hyperedges containing some edge in $L$ is at least $LM/\binom{k}{2}$. Since $\cH$ is assumed to be edge-minimal, it follows that every hyperedge $h$ contains a vertex pair that is only contained in $h$. Hence $|E(\cH)| \leq \binom{n}{2}$. It follows that 
$$LM/\binom{k}{2} \leq \abs{E(\cH)} \leq \binom{n}{2},$$
which implies that 
$$L \leq \frac{k^2}{2M} \binom{n}{2} \leq \frac{\epsilon}{2}\binom{n}{2}.$$
This completes the proof of the claim.
\end{proof}

Let $G'$ be the blowup of $G$ by a factor of $b = Mv(G)^2k$, i.e., $G' = G(b)$.
Suppose $V(G) = \{v_1, \ldots, v_s\}$ and $V_i$ is the blowed-up independent set in $G'$ that corresponds to $v_i$.
Recall the celebrated Erd\H{o}s-Stone-Simonovits theorem \cite{ES1,ES2}, which states that
for a fixed simple graph $F$, $ex(n,F) = \lp 1 -\frac{1}{\chi(F)-1}+o(1) \rp \binom{n}{2}$. Since $\chi(G') = \chi(G)$, it follows by the Erd\H{o}s-Stone-Simonovits theorem that for sufficiently large $n$, $H'$ contains $G'$ as a subgraph. 

Our goal is to give an embedding $f$ of $G$ into $G'$ so that $f(v_i) \in V_i$ for all  $1\leq i \leq s$ and every edge of $G$ is embedded in a distinct hyperedge in $\cH$. For ease of reference, set $L_j = \{v_1, \ldots, v_j\}$. For $1\leq t\leq s$ and $v\in V(G)$, set $N_t(v) = N_G(v) \cap L_t$.
For $i = 1$, just embed $v_1$ to an arbitrary vertex in $V_1$. Suppose that $v_1, \ldots, v_t$ are already embedded and edges in $G[L_t]$ are already embedded in distinct hyperedges. We now want to embed $v_{t+1}$ into an appropriate vertex in $V_{t+1}$, i.e., we want to find a vertex $u \in V_{t+1}$ such that there are distinct unused hyperedges embedding the edges from $u$ to $f(N_t(v_{t+1}))$. Note that each vertex $u$ in $V_{t+1}$ is adjacent to all vertices in $f(N_t(v_{t+1}))$ in $G'$. Let $S_t(u) = \{u\} \times f(N_t(v_{t+1}))$, i.e., $S_t(u)$ is the set of vertex pairs which contain $u$ and another vertex in $f(N_t(v_{t+1}))$.

Recall that $|V_{t+1}| = Mv(G)^2k$. At most $e(G)(k-2)$ vertices in $V_{t+1}$ are contained in hyperedges that are already used. For any of the remaining vertices $u \in V_{t+1}$, if there are no distinct hyperedges embedding all vertex pairs in $S_t(u)$, that means some hyperedge contains at least two vertex pairs $u w_1, u w_2$ in $S_t(u)$. Note that $d_{H'}(\{w_1, w_2\}) \leq M$ by the definition of $H'$. Thus the number of vertices $u\in V_{t+1}$ such that there exists some hyperedge containing at least two vertex pairs in $S_t(u)$ is at most 
$$ \binom{t}{2}M (k-2)  \leq \frac{Mv(G)^2k}{2}.$$
Since $|V_{t+1}| = Mv(G)^2k$, it follows that there exists some $u \in V_{t+1}$ such that $u$ is not contained in any hyperedge already used and there is no hyperedge containing at least two vertex pairs in $S_t(u)$.  It follows that there are distinct unused hyperedges containing all vertex pairs in $S_t(u)$. Set $f(v_{t+1})$ to be this $u$.

By induction, we can then conclude that $\cH$ contains a Berge copy of $G$. This completes the proof of Theorem \ref{thm:general-upper}.
\end{proof}

%%%%%%%%%%%%%%%%%%%%%%%%%%%%%%%%%%%%%%%%%%%%%%
% Proof of Degenerate case
%%%%%%%%%%%%%%%%%%%%%%%%%%%%%%%%%%%%%%%%%%%%%%

\section{Proof of Theorem \ref{thm:degenerate}}

\subsection{Regularity Lemma}
The proof of Theorem \ref{thm:degenerate} uses the Szemer\'edi Regularity Lemma. Given a graph $G$, and two disjoint vertex sets $X,Y \subseteq V(G)$, let $e(X,Y)$ denote the number of edges intersecting both $X$ and $Y$. Define $d(X,Y) = e(X,Y)/|X||Y|$ as the \emph{edge density} between $X$ and $Y$. $(X,Y)$ is called $\epsilon$-regular if for all $X'\subseteq X$, $Y' \subseteq Y$ with $|X'| \geq \epsilon |X|$ and $|Y'| \geq \epsilon |Y|$, we have $\abs{d(X,Y)-d(X',Y')} \leq \epsilon$. We say a vertex partition $V = V_0 \cup V_1 \cup \cdots \cup V_k$ \emph{equipartite} (with the \emph{exceptional set} $V_0$) if $|V_i| = |V_j|$ for all $i,j \in [k]$. The vertex partition $V = V_0 \cup V_1 \cup \cdots \cup V_k$ is said to be $\epsilon$-regular if all but at most $\epsilon k^2$ pairs $(V_j, V_j)$ with $1\leq i < j \leq k$ are $\epsilon$-regular and $|V_0| \leq \epsilon n$. The extremely powerful Szemer\'edi's regularity lemma states the following:
\begin{theorem}\cite{Szemeredi}\label{thm:regLemma}
For every $\epsilon$ and $m$, there exists $N_0$ and $M$ such that every graph $G$ on $n\geq N_0$ admits an $\epsilon$-regular partition $V_0 \cup V_1 \cup \cdots \cup V_k$ satisfying that $m\leq k\leq M$.
\end{theorem}

A $\epsilon$-regular pair satisfies the following simple lemma.
\begin{lemma}\label{lem:large-neighbor}
Suppose $(X,Y)$ is an $\epsilon$-regular pair of density $d$. Then for every $Y'\subseteq Y$ of size $|Y'| \geq \epsilon |Y|$, there exists less than $\epsilon |X|$ vertices in $X$ that have less than $(d-\epsilon)|Y'|$ neighbors in $Y'$.
\end{lemma}
\begin{proof}
Let $Y'\subseteq Y$ with $|Y'| \geq \epsilon |Y|$. Let $X'$ be the set of vertices of $X$ that have less than $(d-\epsilon)|Y'|$ neighbors in $Y'$. Note that $d(X',Y') < (d-\epsilon)$, which can only happen if $|X'| < \epsilon|X|$.
\end{proof}

Using Lemma \ref{lem:large-neighbor}, we will show the following lemma using the standard embedding technique.

\begin{lemma}\label{lem:clique-large-neighbor}
Fix a positive integer $s$. Suppose $(X,Y)$ is an $\epsilon$-regular pair of density $d$ such that $\epsilon \leq 1/4s$, $(d-\epsilon)^s \geq 4\epsilon$ and $|X|,|Y|\geq 4s/(d-\epsilon)^s$. Then there exist disjoint subsets $A,C \subseteq X$ and $B,D \subseteq Y$ such that $|A| = |B| = s$, $|C|\geq \epsilon |X|$, $|D| \geq \epsilon |Y|$, and there is a complete bipartite graph connecting $A$ and $D$, $B$ and $C$ as well as $A$ and $B$.
\end{lemma}
\begin{proof}
Denote $A = \{a_1, \ldots, a_s\}$ and $B = \{b_1, \ldots, b_s\}$. For each $i\in [s]$, we will first embed $a_i$ to $X$ one vertex at a time. 
After embedding the $k^{\textrm{th}}$-vertex, we will show that the following condition is satisfied:
$$\abs*{Y\cap \ds\bigcap_{i=1}^k N(a_i)} \geq (d-\epsilon)^k|Y|.$$
The condition is trivially satisfied when $k = 0$. Suppose that we already embedded the vertices $a_1, \ldots, a_t$ for some $t > 0$. Let $Y_t' =  Y\cap \bigcap_{i=1}^t N(a_i)$. By induction, $|Y_t'| \geq (d-\epsilon)^t|Y| > \epsilon|Y|$. Hence by Lemma \ref{lem:large-neighbor}, at least $\lp(1-\epsilon)|X|-s\rp$ vertices in $X$ have at least $(d-\epsilon)|Y_t'|$ neighbors in $Y_t'$. Embed $a_{t+1}$ to one of these $\lp(1-\epsilon)|X|-s\rp$ vertices and it's easy to see that 
$$\abs*{Y\cap \ds\bigcap_{i=1}^{t+1} N(a_i)} \geq (d-\epsilon)|Y_t'| \geq (d-\epsilon)^{t+1}|Y|.$$
Now we want to embed $b_i$ to $Y_s'$ one vertex at a time. The process is entirely the same as long as $$(d-\epsilon)^{s}(|X|-s) \geq \epsilon|X|$$ and $$(d-\epsilon)^s|Y|-\epsilon|Y|-s\geq 1,$$
which are satisfied by our assumption on $d$, $|X|$ and $|Y|$.

\end{proof}

\subsection{Constructions for Theorem \ref{thm:degenerate}}\label{sec:lower-bound-1}
Before we prove Theorem \ref{thm:degenerate}, we first give two constructions and show that if $G$ does not satisfy the conditions $(1)$ and $(2)$ in Theorem \ref{thm:degenerate}, then at least one of the constructions do not contain a Berge copy of $G$. In particular, suppose $A, B$ are two disjoint set of vertices enumerated as $A = \{a_1, \dots, a_{n/2}\}$ and $B = \{b_1, \dots, b_{n/2}\}$. Let $\cH_1$ be a $3$-uniform hypergraph such that $V(\cH_1) = A \cup B$ and $E(\cH_1) = \{\{a_i, b_j, b_{j+1}\}: \textrm{ $j$ is odd}\}$. Let $\cH_2$ be a $3$-uniform hypergraph such that $V(\cH_2) = A \cup B$ and $E(\cH_2) = \{\{b_1, a_i, b_{j}\}: a_i \in A, b_j \in B \backslash \{b_1\}\}$. Observe that 
$$\lim_{n\to \infty} \frac{|E(\partial(\cH_1))|}{\binom{n}{2}}  = \ds\lim_{n\to\infty} \frac{|E(\partial(\cH_2))|}{\binom{n}{2}} = \frac{1}{2}.$$

\begin{claim}\label{cl:forward}
If $\td_3(G) = 0$, then condition $(1)$ and $(2)$ of Theorem \ref{thm:degenerate} must hold.
\end{claim}
\begin{proof}
Suppose that $\td_3(G) = 0$. We claim that $(1)$ and $(2)$ must hold. 
First observe that $\cH_1$ contains no Berge triangle. Hence $G$ must be triangle-free otherwise $\cH_1$ is Berge-$G$-free. Now note that given a hypergraph $\cH$, if $\partial(\cH)$ is $G$-free, then $\cH$ must be Berge-$G$-free. Observe that $\partial(\cH_1)$ contains a bipartite subgraph $B \subseteq \partial(\cH_1)$ such that $E(\partial(\cH_1))-E(B)$ is a matching (possibly empty) in one of the partition of $B$.
Hence if there is no such bipartite subgraph in $G$, then $\partial(\cH_1)$ is $G$-free, implying that $\cH_1$ is Berge-$G$-free. Since $\td_3(G) = 0$, it follows that $G$ must satisfy condition $(2)$. Similarly, observe that $\partial(\cH_2)$ satisfies condition $(1)$. Hence if $G$ doesn't satisfy condition $(1)$, then $\cH_2$ is Berge-$G$-free, which contradicts that $\td_3(G) = 0$. Therefore we can conclude that $(1)$ and $(2)$ must hold for $G$.
\end{proof}

\subsection{Proof of Theorem \ref{thm:degenerate}}
The forward direction is proved in Claim \ref{cl:forward}.
It remains to show that if $G$ satisfies the conditions $(1)$ and $(2)$ in Theorem \ref{thm:degenerate}, then $\td_3(G) = 0$. Suppose not, i.e., $\td_3(G) \geq d$ for some $d > 0$. Our goal is to show that for every $3$-graph $\cH$ on (sufficiently large) $n$ vertices and at least $d \binom{n}{2}$ edges in $\partial(\cH)$,  $\cH$ contains a Berge copy of $G$. 

Assume first that $\cH$ is edge-minimal while maintaining the same shadow. It follows that in every hyperedge $h$ of $\cH$, there exists some $e \in \binom{h}{2}$ such that $e$ is contained only in $h$. Moreover, note that since each hyperedge covers at most $3$ edges in $\partial(\cH)$, we have that $$|E(\cH)| \geq \frac{1}{3}|E(\partial(\cH))| \geq \frac{d}{3}\binom{n}{2}.$$
Call an edge $e \in \partial(\cH)$ \emph{uniquely embedded} if there exists a unique hyperedge $h \in E(\cH)$ containing $e$. Now randomly partition $V(\cH)$ into three sets $X,Y,Z$ of the same size. Let $e(X,Y,Z)$ denote the number of hyperedges of $\cH$ intersecting each of the sets $X,Y,Z$ on at most one vertex. It's easy to see that 
$\textrm{E}[e(X,Y,Z)] = \frac{2}{9} |E(\cH)|$. Hence there exists a $3$-partite subhypergraph $\cH_1 = X\cup Y\cup Z$ of $\cH$ such that $|E(\cH_1)| \geq \frac{2}{9} |E(\cH)|$. Note that each hyperedge $h$ of $\cH_1$ contains some $e \in \binom{h}{2}$ that is uniquely embedded. Hence there are at least $\frac{2}{9}|E(\cH)|$ uniquely embedded edges in $\partial(\cH_1)$. Without loss of generality, assume that there are at least $\frac{2}{27}|E(\cH)|$ uniquely embedded edges between the vertex sets $X$ and $Y$ in $\partial(\cH_1)$. Let $\cH'$ be the subhypergraph of $\cH_1$ with only hyperedges containing a uniquely embedded edge between $X$ and $Y$. 

For ease of reference, let $H' = \partial(\cH')$ and let $H'[X\cup Y]$ be the subgraph of $\partial(\cH')$ induced by $X\cup Y$. Note that $H'[X\cup Y]$ is bipartite with at least $\frac{2}{27}|E(\cH)| \geq \frac{2d}{81}\binom{n}{2} = d'\binom{n}{2}$ edges.

Let $\epsilon = \epsilon(s,d'/2)$ be small enough so that $\epsilon$ satisfies the assumptions in Lemma \ref{lem:clique-large-neighbor}. Applying the regularity lemma on $H'[X\cup Y]$, we can find an $\epsilon$-regular partition in which there exist two parts $X'\subseteq X, Y'\subseteq Y$ such that $(X', Y')$ is an $\epsilon$-regular pair with edge density at least $d'/2$. Moreover, $|X'|,|Y'| \geq n/M$ for some constant $M > 0$. By Lemma \ref{lem:clique-large-neighbor}, we can find 
disjoint subsets $A,C \subseteq X'$ and $B,D \subseteq Y'$ such that $|A| = |B| = 2s$, $|C|\geq \epsilon|X'|$, $|D| \geq \epsilon |Y'|$, and there is a complete bipartite graph connecting $A$ and $D$, $B$ and $C$ as well as $A$ and $B$.

Now consider the subhypergraph $\hh =\cH'[C\cup D\cup Z]$ of $\cH'$ induced by the vertex set $C\cup D\cup Z$, i.e., all hyperedges in $\hh$ contain vertices only in $C\cup D\cup Z$. Given a vertex set $S \subseteq V(\hh)$, define $\hat{d}_S(v)$ as the number of neighbors of $v$ in $S$ in $\partial(\hh)$. 

\begin{claim}\label{cl:1p4clus}
If there exists some $z \in Z$ such that $\hd_C(v) \geq 2s$ and $\hd_D(v) \geq 2s$, then $\cH'$ contains a Berge-$C_5(1,s,s,s,s)$ as subhypergraph.
\end{claim}
\begin{proof}
Denote the $C_5(1,s,s,s,s)$ that we wish to embed as $\{v_1\} \cup V_2\cup V_3\cup V_4\cup V_5$. Let $v_1= z$. Let $C_z, D_z$ be the set of neighbors of $z$ in $C$ and $D$ respectively in $\partial(\hh)$. We wish to embed $V_2$ in $C_z$, $V_3$ in $B$, $V_4$ in $A$ and $V_5$ in $D_z$. 
Note that $|C_z|,|D_z|\geq 2s$ by our assumption. Pick arbitrary $s$ of them to be $V_2$. For each vertex pair $\{z,w\}$ where $w \in V_2$, there exists a hyperedge $h \subseteq C \cup D \cup Z$ containing $\{z,w\}$. Use $h$ to embed $\{z,w\}$. Observe that at most $s$ vertices in $D_z$ or $B$ are contained in hyperedges embedding the edges from $z$ to $V_2$. Since $|D_z|\geq 2s$, we can set $V_5$ to be arbitrary $s$ vertices among vertices in $D_z$ that are not contained in any hyperedge embedding the edges from $z$ to $V_2$. Similarly, since $|A|, |B| \geq 2s$, we can set $V_3$ and $V_4$ to be arbitrary $s$ vertices among vertices in $B$ and $A$ that are not contained in any hyperedge embedding the edges from $z$ to $V_2$ and from $z$ to $V_5$ respectively. We then have distinct hyperedges (in $\hh$ only) embedding the edges from $z$ to $V_2$ and $z$ to $V_5$, $V_2$ to $V_3$ and $V_4$ to $V_5$ respectively. Moreover, recall that by our choice of $X'$ and $Y'$, vertex pairs between $V_4$ and $V_5$ are uniquely embedded (with the third vertex in $Z$), i.e.,  there exist distinct hyperedges embedding them. Hence, we obtain a Berge-$C_5(1,s,s,s,s)$ in $\cH'$.
\end{proof}

Now observe that $|C|\geq \epsilon |X'|$, $|D| \geq \epsilon |Y'|$. Hence by the $\epsilon$-regularity of $(X',Y')$, the number of edges $e(C,D)$ in $\partial(\hh)$ satisfies that 
$$e(C,D) \geq (\frac{d'}{2}-\epsilon)|C||D| \geq (\frac{d'}{2}-\epsilon)\epsilon^2|X'||Y'|\geq (\frac{d'}{2}-\epsilon)\epsilon^2 \frac{n^2}{M^2} = cn^2$$
where $c$ is a constant depending on $\epsilon$ and $d'$.

\begin{claim}\label{cl:matching3clus}
If $\cH'$ contains no Berge-$C_5(1,s,s,s,s)$ as subhypergraph, it must contain a Berge-$F$ where $F$ is any triangle-free subgraph of $C_3(s,s,s;$ $\{\{1,2\}\})$.
\end{claim}
\begin{proof}
By claim \ref{cl:1p4clus}, since $\cH'$ contains no Berge-$C_5(1,s,s,s,s)$ as subhypergraph, it follows that given any $v\in Z$, one of $\hd_C(v)$, $\hd_D(v)$ must be smaller than $2s$. Let $Z_1$ be the set of vertices $z \in Z$ with $\hd_C(v) < 2s$, and $Z_2$ be the set of vertices $z \in Z$ with $\hd_D(v) < 2s$. Let $e(Z_1, D)$ and $e(Z_2, C)$ denote the number of edges between $Z_1$ and $D$, $Z_2$ and $C$ respectively in $\partial(\hh)$. Since $e(C,D) \geq cn^2$ and all hyperedges in $\hh$ contains a vertex in $Z$, it follows that at least one of $e(Z_1, D)$ and $e(Z_2, C)$ must be at least $\Omega(n^2)$. WLOG, suppose $e(Z_1, D) \geq c'n^2$ for some $c' > 0$. Recall the classical result of K\H{o}v\'ari, S\'os and Tur\'an \cite{KST}, who showed that $ex(n, K_{r,t}) = O(n^{2-1/r})$ where $r\leq t$. By the Tur\'an number of complete bipartite graphs, we have that for sufficiently large $n$, $\partial(\hh)[D\cup Z_1]$ contains a complete bipartite graph $K_{(2s)^{s+1},(2s)^{s+1}}$. For ease of reference, call this complete bipartite graph $K$.

Let $F$ be an arbitrary triangle-free subgraph of $C_3(s,s,s;$ $\{\{1,2\}\})$. We now show that $\hh$ contains a Berge-$F$ subhypergraph. Let $C_1$ be the collection of vertices $v$ in $C$ such that there is some hyperedge containing $v$ and one of the edges in $K$.
Observe that for each $v\in C_1$, $\hd_{Z_1 \cap K}(v) \leq s$, otherwise we obtain a Berge-$C_5(1,s,s,s,s)$ in $\cH'$. Moreover, recall that for every $v \in Z_1$, $\hd_C(v) < 2s$.
It follows that there must be an edge $x_1 y_1 \in \partial(\hh)$ with $x_1 \in C_1, y_1\in Z_1$ such that at least $(2s)^{s}$ vertices in $D \cap K$ form a hyperedge containing $x_1y_1$. Now consider the subgraph  $K'$ of $K$ induced by these $(2s)^{s}$ vertices in $D\cap K$ as well as the non-neighbors of $x_1$ in $Z_1 \cap K$. Observe that $K'$ is also a complete bipartite graph with at least $(2s)^{s}$ vertices in each partition. Hence by the same logic, we can find another edge $x_2 y_2 \in \partial(\hh)$ with $x_2 \in C_1, y_2 \in Z_1$ such that at least $(2s)^{s-1}$ vertices in $D\cap K'$ form hyperedges containing $x_1 y_1$ and $x_2 y_2$ respectively. Continuing this process $s$ steps, it is not hard to see that we can find a Berge-$F$ subhypergraph in $\hh$.

In summary, if $\cH$ is $3$-graph with at least $d\binom{n}{2}$ edges in $\partial(\cH)$ for some $d> 0$ and $n$ sufficiently large, then $\cH$ contains  either a Berge-$C_5(1,s,s,s,s)$ or a Berge-$F$ where $F$ is any triangle-free subgraph of $C_3(s,s,s;\{\{1,2\}\})$. Moreover, observe that if $G$ satisfies the conditions $(1)$ and $(2)$ in Theorem \ref{thm:degenerate}, then $G$ is a subgraph of both $C_5(1,s,s,s,s)$ and $C_3(s,s,s;\{\{1,2\}\})$. Hence it follows that $\td_3(G) = 0$. This completes the proof of the theorem.

\end{proof}

It is easy to see that Theorem \ref{thm:degenerate} implies Corollary \ref{cor:degenerate-characterize1}. In the remaining of this section, we show that Corollary \ref{cor:degenerate-characterize1} and Corollary \ref{cor:degenerate-characterize2} are indeed equivalent.

\begin{proof}[Proof of Corollary \ref{cor:degenerate-characterize2}]
It suffices to show that a graph $G$ is contained in both $C_5(1,s,s,s,s)$ and $C_3(s,s,s;$ $\{\{1,2\}\})$ (for some $s$) if and only if $G$ is a subgraph of one of the graphs in Figure \ref{fig:embed3}. We follow the labelling in Figure \ref{fig:embed2}.
The backward direction is easy. For the forward direction, there are two cases:

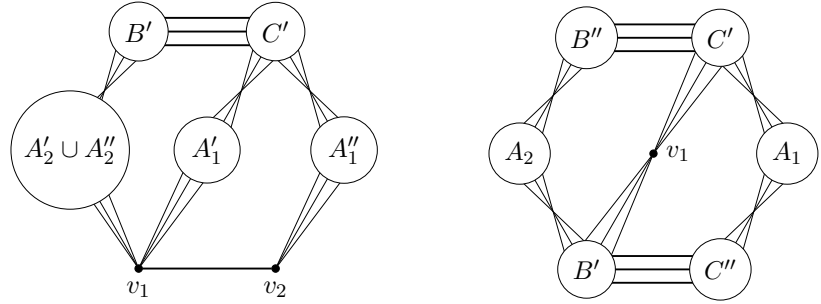
\begin{figure}[htb]
	\begin{center}
        \begin{minipage}{.2\textwidth}
        		\resizebox{5cm}{!}{\begin{tikzpicture}[scale=1, Wvertex/.style={circle, draw=black, fill=white, scale=1}, bvertex/.style={circle, draw=black, fill=black, scale=0.3}]

\node [bvertex, label=below:$v_1$] (v1) at (-1, -1.732) {};
\node [bvertex,label=below:$v_2$] (v2) at (1, -1.732) {};
\node [Wvertex] (v0) at (0,0) {};
\node [Wvertex] (v3) at (2,0) {};
\node [Wvertex] (v4) at (1,1.732) {};
\node [Wvertex] (v5) at (-1,1.732) {};
\node [Wvertex] (v6) at (-2,0) {};

\draw[thick] (v1)  -- (v2);
\draw (v1)-- ($(v0) + (-0.2,0)$);
\draw (v1)-- ($(v0) + (0,0)$);
\draw (v1)-- ($(v0) + (0.3,0)$);

\draw (v2)-- ($(v3) + (-0.2,0)$);
\draw (v2)-- ($(v3) + (0,0)$);
\draw (v2)-- ($(v3) + (0.3,0)$);

\draw (v1)-- ($(v6) + (-0.2,0)$);
\draw (v1)-- ($(v6) + (0,0)$);
\draw (v1)-- ($(v6) + (0.3,0)$);

\draw[thick]  ($(v5)+(0,0.2)$) -- ($(v4) + (0,0.2)$);
\draw[thick] (v5) -- (v4);
\draw[thick] ($(v5)+(0,-0.2)$) -- ($(v4) + (0,-0.2)$);

\draw  ($(v5)+(-0.2,0.2)$) -- ($(v6) + (0.2,-0.2)$);
\draw (v5) -- (v6);
\draw  ($(v5)+(0.2,-0.2)$) -- ($(v6) + (-0.2, 0.2)$);

\draw  ($(v0)+(-0.2,0.2)$) -- ($(v4) + (0.2,-0.2)$);
\draw (v0) -- (v4);
\draw  ($(v0)+(0.2,-0.2)$) -- ($(v4) + (-0.2, 0.2)$);

\draw  ($(v4)+(-0.2,-0.2)$) -- ($(v3) + (0.2,0.2)$);
\draw (v4) -- (v3);
\draw  ($(v4)+(0.2,0.2)$) -- ($(v3) + (-0.2, -0.2)$);

\node [Wvertex] (v0) at (0,0) {$A_1'$};
\node [Wvertex] (v3) at (2,0) {$A_1''$};
\node [Wvertex] (v4) at (1,1.732) {$C'$};
\node [Wvertex] (v5) at (-1,1.732) {$B'$};
\node [Wvertex] (v6) at (-2,0) {$A_2' \cup A_2''$};
\end{tikzpicture}}
        \end{minipage}
            \hspace{3cm}
        \begin{minipage}{.2\textwidth}
        		\resizebox{4.5cm}{!}{\begin{tikzpicture}[scale=1, Wvertex/.style={circle, draw=black, fill=white, scale=1}, bvertex/.style={circle, draw=black, fill=black, scale=0.3}]

\node [Wvertex] (v1) at (-1, -1.732) {};
\node [Wvertex] (v2) at (1, -1.732) {};
\node [bvertex, label=right:$v_1$] (v0) at (0,0) {};
\node [Wvertex] (v3) at (2,0) {};
\node [Wvertex] (v4) at (1,1.732) {};
\node [Wvertex] (v5) at (-1,1.732) {};
\node [Wvertex] (v6) at (-2,0) {};

\draw (v0)-- ($(v1) + (-0.2,0.2)$);
\draw (v0)-- ($(v1) + (0,0)$);
\draw (v0)-- ($(v1) + (0.2,-0.2)$);

\draw (v0)-- ($(v4) + (0.2,-0.2)$);
\draw (v0)-- ($(v4) + (0,0)$);
\draw (v0)-- ($(v4) + (-0.2,0.2)$);

\draw[thick]  ($(v5)+(0,0.2)$) -- ($(v4) + (0,0.2)$);
\draw[thick] (v5) -- (v4);
\draw[thick] ($(v5)+(0,-0.2)$) -- ($(v4) + (0,-0.2)$);

\draw[thick]  ($(v1)+(0,0.2)$) -- ($(v2) + (0,0.2)$);
\draw[thick] (v1) -- (v2);
\draw[thick] ($(v1)+(0,-0.2)$) -- ($(v2) + (0,-0.2)$);

\draw  ($(v5)+(-0.2,0.2)$) -- ($(v6) + (0.2,-0.2)$);
\draw (v5) -- (v6);
\draw  ($(v5)+(0.2,-0.2)$) -- ($(v6) + (-0.2, 0.2)$);

\draw  ($(v3)+(-0.2,0.2)$) -- ($(v2) + (0.2,-0.2)$);
\draw (v3) -- (v2);
\draw  ($(v3)+(0.2,-0.2)$) -- ($(v2) + (-0.2, 0.2)$);

\draw  ($(v4)+(-0.2,-0.2)$) -- ($(v3) + (0.2,0.2)$);
\draw (v4) -- (v3);
\draw  ($(v4)+(0.2,0.2)$) -- ($(v3) + (-0.2, -0.2)$);

\draw  ($(v6)+(-0.2,-0.2)$) -- ($(v1) + (0.2,0.2)$);
\draw (v6) -- (v1);
\draw  ($(v6)+(0.2,0.2)$) -- ($(v1) + (-0.2, -0.2)$);

\node [Wvertex] (v1) at (-1, -1.732) {$B'$};
\node [Wvertex] (v2) at (1, -1.732) {$C''$};
%\node [bvertex] (v0) at (0,0) {};
\node [Wvertex] (v3) at (2,0) {$A_1$};
\node [Wvertex] (v4) at (1,1.732) {$C'$};
\node [Wvertex] (v5) at (-1,1.732) {$B''$};
\node [Wvertex] (v6) at (-2,0) {$A_2$};

\end{tikzpicture}}
        \end{minipage}
    \end{center} 
    \caption{Equivalence of characterizations in Corollary \ref{cor:degenerate-characterize1} and \ref{cor:degenerate-characterize2}. }
    \label{fig:embed4}
\end{figure}

\begin{description}

\item Case 1: With loss of generality, $v_1$ is in $B$. 
Let $v_2\in C$ be the vertex matched to $v_1$. 
Let $B'=B\setminus \{v_1\}$, and $C'=C\setminus \{v_2\}$. Note that $G-v_1$ is a bipartite graph, i.e., $V(G)-v_1 = U_1 \cup U_2$. With loss of generality, we can assume
$B'\subseteq U_1$, $C'\subseteq U_2$ and $v_2 \in U_2$ by properly swapping two ends of the matching edges between $B$ and $C$ if needed.

Since $G- v_1$ is bipartite, the vertex set $A$ is partitioned into two parts $A_1 \subseteq U_1, A_2\subseteq U_2$. Let $A_1', A_2'$ be the neighbors of $v$ in $A_1, A_2$ respectively, $A_1'', A_2''$ be the non-neighbors of $v$ in $A_1, A_2$ respectively. Recall that $v_2\in U_2$. It follows that $v_2$ is independent with $A_2' \cup A_2''$. Moreover, since $G$ is triangle-free, $v_2$ is also independent with $A_1'$.  

It then follows that $G$ can be embedded into the first graph of Figure \ref{fig:embed3} in the same way labelled in Figure \ref{fig:embed4} (note that there are no edges between $v_1$ and $A_2''$).

\item Case 2: $v_1$ is in $A$. 
Since $G-v_1$ is bipartite, we can write $V(G)-v_1 = U_1 \cup U_2$.
WLOG, assume that $B\subseteq U_1$ and $C\subseteq U_2$ by properly swapping two ends of the matching edges between $B$ and $C$ if needed. Moreover, write $A = A_1 \cup A_2\cup \{v\}$ where $A_1 \in U_1$ and $A_2 \in U_2$. Write $B = B' \cup B''$, $C =C' \cup C''$ such that $B'$ and $C'$ are the neighbors of $v_1$ in $B$ and $C$ respectively. Since $G$ is triangle-free, it follows that $v_1$ is independent with $B''$ and $C''$.

It then follows that $G$ can be embedded into the second graph of Figure \ref{fig:embed3} in the same way labelled in Figure \ref{fig:embed4}.

\end{description}
\end{proof}

\begin{section}{Proof of Theorem \ref{thm:cover-density-3}}

If $\chi(G) \geq 4$, we are done by Theorem \ref{thm:general}.
If $\chi(G) \leq 3$ and $G$ is not degenerate, the two hypergraphs we constructed in Section \ref{sec:lower-bound-1} provide the lower bound $1/2$, which is also an upper bound by Theorem \ref{thm:general-upper}. Theorem \ref{thm:degenerate} resolves the case when $G$ is degenerate.

\end{section}

%%%%%%%%%%%%%%%%%%%%%%%%%%%%%%%%%%%%%%%%%%%%%%
% Bibliography
%%%%%%%%%%%%%%%%%%%%%%%%%%%%%%%%%%%%%%%%%%%%%%

\end{document}